\documentclass[10pt]{amsart}

\usepackage{fullpage}
\usepackage{amscd,amsmath,amsfonts,amsthm,amssymb}
\usepackage[all]{xy}
\usepackage{enumitem}
\usepackage{hyperref}
\usepackage{cleveref}
\usepackage[top=30truemm,bottom=30truemm,left=25truemm,right=25truemm]{geometry}

\hypersetup{colorlinks=true,citecolor=blue,linkcolor=magenta,urlcolor=green, bookmarksnumbered=true,pdfborder={1 0 0},bookmarkstype=toc}

\theoremstyle{definition}
\newtheorem{thm}{Theorem}[section]
\crefname{thm}{Theorem}{Theorem}
\newtheorem{dfn}[thm]{Definition}
\crefname{dfn}{Definition}{Definition}
\newtheorem{lem}[thm]{Lemma}
\crefname{lem}{Lemma}{Lemma}
\newtheorem{pro}[thm]{Proposition}
\crefname{pro}{Proposition}{Proposition}
\newtheorem{cor}[thm]{Corollary}
\crefname{cor}{Corollary}{Corollary}
\newtheorem{exa}[thm]{Example}
\crefname{exa}{Example}{Example}
\newtheorem{rmk}[thm]{Remark}
\crefname{rmk}{Remark}{Remark}
\newtheorem{nota}[thm]{Notation}
\crefname{nota}{Notation}{Notation}

\newcommand{\exend}{\unskip\nobreak\hfill$\blacklozenge$}

\newcommand{\F}{\mathbb{F}}
\newcommand{\Z}{\mathbb{Z}}
\newcommand{\Hom}{\mathop{\mathrm{Hom}}\nolimits}
\newcommand{\ilim}[1][]{\mathop{\varinjlim}\limits_{#1}}
\newcommand{\plim}[1][]{\mathop{\varprojlim}\limits_{#1}}
\renewcommand{\lim}[1][]{\mathop{\mathrm{lim}}\limits_{#1}}
\newcommand{\colim}[1][]{\mathop{\mathrm{colim}}\limits_{#1}}
\newcommand{\Cond}{\mathrm{Cond}}
\newcommand{\Ring}{\mathrm{Ring}}
\newcommand{\Perf}{\mathrm{Perf}}
\newcommand{\Stone}{\mathrm{Stone}}

\begin{document}

\title[A new approach to $\delta$-rings via Stone duality]{A new approach to $\delta$-rings via Stone duality}

\author[Y. Yamada]{Yuto Yamada}
\address{Department of Mathematics, Institute of Science Tokyo, 2-12-1 Ookayama, Meguro, Tokyo 152-8551}
\email{yamada.y.f243@m.isct.ac.jp}


%\subjclass{13}
%\subjclass[2000]{Primary 13-XX}
%\subjclass[2000]{Primary ; Secondary}
%\date{\today \, (\printtime)}
%\date{\today}

\maketitle

\begin{abstract}
We define "Stone $\delta$-rings" as a new class of $\delta$-rings. As an analogue of \cite{Gre24}, we show that the new class relates light condensed mathematics via Stone duality. Also, we examine some phenomena for this relationship, for example, we observe $\delta$-rings which corresponds to metrizable compact Hausdorff spaces.
\end{abstract}

\tableofcontents

\section{Introduction}

In \cite{Ant23}, Antieau defined the notion of \emph{Stone algebras}, and clarified the relationship between them and profinite sets (\cite[Lemma 3.5]{Ant23}): For any ring $k$ whose idempotent elements are only $\{0,1\}$, there is an equivalence of categories
\[
\Hom_k(-,k):(\Stone_k)^\text{op}\simeq\mathrm{ProFin}:\mathrm{Cont}(-,k),
\]
where $\Stone_k$ denotes the category of "Stone $k$-algebra".

This result obviously generalizes \emph{"classical Stone duality"}, which gives the (anti-)equivalence between Boolean algebras and profinite sets (originally discussed in \cite{Sto34} and \cite{Sto36}).

Profinite sets have some interesting aspects; in particular, we will focus on \emph{light condensed mathematics}, which was introduced by Scholze and Clausen in \cite{AnStk}. (Also, it is a "light" version of "classical" condensed mathematics developed in \cite{LCM}.) Roughly speaking, light condensed mathematics is the theory of sheaves on light profinite sets equipped with the finitary Grothendieck topology defined by finite jointly surjective covers (\cite{AnStk} and \cite[Definition 4.1]{Gre24}):
\[
\Cond(C)^\text{light}=\mathrm{Shv}_\text{eff}(\mathrm{ProFin}^\text{light},C).
\]
In particular, light condensed sets can be considered as a new class of "reasonable" topological spaces. For example, the category of metrizable compact Hausdorff spaces can be embedded into the category of light condensed sets:
\[
\mathrm{CHaus}^\text{light}\xrightarrow{\simeq}\Cond(\mathrm{Set})_\text{qcqs}^\text{light}\subset\Cond(\mathrm{Set})^\text{light}.
\]
Moreover, for light condensed abelian groups  (i.e. abelian group objects in light condensed sets), there is the \emph{condensed cohomology} theory such that it coincides with the sheaf cohomology when it is restricted to metrizable compact Hausdorff spaces $K$ (\cite{AnStk} and \cite[Proposition 2.3.7]{Solid}):
\[
H_{\text{cond}}^i(K,\Z)\simeq H_{\text{sheaf}}^i(K,\Z).
\]\indent
In \cite{Gre24}, Gregoric examined the relationship between fpqc sheaves and "classical" condensed sets via the result of \cite{Ant23}. For example, for any field $k$, the faithfully flatness in Stone algebras is identified with the surjectivity in profinite sets via Stone duality (\cite[Proposition 1.19]{Gre24}), and there is an adjunction
\[
\xymatrix{
{\Cond(\mathrm{Set})}\ar@<1.2ex>[rr]_-{\bot}&&{\mathrm{Shv}_\text{fpqc}(\Ring_k,\mathrm{Set}),}\ar@<1.2ex>[ll]
}
\]
noting that we equip $\Ring_k$ with the fpqc topology (\cite[Theorem C]{Gre24}).

For a fixed prime $p$, the author is interested in the relationship for $k=\F_p$, and proceeds the discussion for $\delta$-rings. In fact, (a part of) the theory of $\F_p$-algebras can be considered as the theory of ($p$-complete) $\delta$-rings via the Witt vector functor. Roughly speaking, $\delta$-rings can be regarded as the rings that admit a "Frobenius lift". 

Furthermore, Bhatt-Scholze developed \emph{prisms} and \emph{prismatic cohomology}, which have deep information about the $p$-adic Hodge theory, (integral) perfectoids ring, etc. However, in this article, we do not discuss the aspects, only the relation to characteristic $p$.

As the analogue of Stone algebras, the author defines \emph{Stone $\delta$-rings} as a new class of $\delta$-rings, which corresponds to the notion of profinite sets. Indeed, the notion gives the following isomorphism of sites.

\begin{thm}[{\cref{thm:Stone duality 2}} and \cref{thm:comparison of sites 2}]
There exists an isomorphism of sites:
\[
\mathrm{Spec}((-)/p):((\Stone^{\delta,\land})^\text{op},\tau_{p\text{-fpqc}})\simeq(\mathrm{ProFin},\tau_\text{eff}):\mathrm{Cont}(-,\Z_p),
\]
where $\tau_{p\text{-fpqc}}$ denotes the Grothendieck topology defined by $p$-completely faithfully flat covers.
\end{thm}

Thanks to this comparison, we can compare light condensed sets with "$\delta$-stacks". The notion is the main target of this article, which the author expect to reflect the "topological" aspects of (general) $p$-complete $\delta$-rings for the future works. We use the notion of "continuous functors" as in \cite[\href{https://stacks.math.columbia.edu/tag/00WU}{[00WU]}]{Sta} and obtain the following relationship between light condensed sets and $\delta$-rings in terms of adjunctions:

\begin{thm}[{\cref{cor:morphism of topoi}}]
We obtain the following adjunction:
\[
\xymatrix{
{\mathrm{TD}\mathrm{Stk}^{\delta,\land}}\ar@<1.2ex>[rr]^-{\psi^\ast}_-{\bot}&&{\Cond(\mathrm{Set})^\text{light},}\ar@<1.2ex>[ll]^-{\psi_\ast}
}
\]
which is induced by the functor $\Stone^{\delta,\land}\to\mathrm{ProFin}^\text{light}\ ;\ A\longmapsto\mathrm{Cont}(A,\Z_p)$.
\end{thm}

As an example, in \cref{exa:Betti delta-stacks}, we obtain the notion of \emph{"Betti $\delta$-stacks"} which is an analogue of Betti stack defined in \cite{Sch24}.

\subsection*{Notation}

We fix a prime $p$ throughout this article, and the notion $(-)^\land$ means classical $p$-completion, but not derived completion. Also, for any ring $A$, the category $\mathcal{D}(A)$ means the derived category of $A$-modules.

\subsection*{Acknowledgement}

The author would like to thank Kazuki Hayashi, Ryo Ishizuka and Kazuma Shimomoto for their beneficial discussions for $\delta$-rings, Makoto Enokizono and Ryo Suzuki for light condensed mathematics, and Thomas Geisser who was my adviser in the master's degree and gave me chances to study (light) condensed mathematics and $\delta$-rings.

\section{Stone Duality}

In this section, we will review the "new" Stone duality in characteristic $p$, and examine a connection with the usual $\F_p$-algebras.

\subsection{Stone $\F_p$-Algebras}

In \cite{Ant23}, "classical" Stone duality is generalized. Let us describe this generalization.

\begin{dfn}[{\cite[Definition 3.2]{Ant23} and {\cite[Definition 1.1]{Gre24}}}, Stone algebra]\label{dfn:Stone algebra}
Let $k$ be a ring. The category $\Stone_k\subset\Ring_k$ of \emph{Stone $k$-algebras} is the full subcategory of $\Ring_k$ spanned under filtered colimits by $k^S$ for a finite set $S$. (Note that we consider $k^S$ as the ring of continuous maps from $S$ to the discrete ring $k$.)
\exend
\end{dfn}

\begin{thm}[{\cite[Lemma 3.5]{Ant23}} and {\cite[Theorem 1.2 and Proposition 1.10]{Gre24}}, Stone dulality]\label{thm:Stone duality}
Let $k$ be a ring with idempotent elements $\{0,1\}$. The following functors induce an equivalence of categories $\mathrm{ProFin}\simeq\Stone_k^\text{op}$.
\[
\Phi:\mathrm{ProFin}\ni S\longmapsto\mathrm{Cont}(S,k)\in\Stone_k^\text{op},\quad\Psi:\Stone_k^\text{op}\ni A\longmapsto\Hom_k(A,k)\in\mathrm{ProFin}.
\]
Moreover, if $k$ is a field, $\Psi$ can be identified with the functor given by $A\longmapsto\mathrm{Spec}(A)$.
\end{thm}
\begin{proof}
The first statement is \cite[Lemma 3.5]{Ant23} and \cite[Theorem 1.2]{Gre24}. The second statement is \cite[Proposition 1.10]{Gre24}.
\end{proof}

\subsection{Perfect $\F_p$-Algebras}

We will discuss the relationship between perfect $\F_p$-algebras and Stone $\F_p$-algebras.

\begin{dfn}[$p$-Boolean algebra]\label{dfn:p-Bool}
An $\F_p$-algebra $A$ is a \emph{$p$-Boolean algebra} if for all $a\in A$, we have $a^p=a$. We let $\mathrm{Bool}_{\F_p}$ denote the category of $p$-Boolean algebras.
\exend
\end{dfn}

Indeed, Stone $\F_p$-algebras discussed in the previous subsection are $p$-Booloean algebras.

\begin{pro}[{\cite[Proposition 1.26]{Gre24}}]\label{pro:p-Bool}
The notion of Stone $\F_p$-algebras coincides with the notion of $p$-Boolean algebras.
\end{pro}

\begin{lem}[{\cite[Remark 1.27 and Corollary 1.28]{Gre24}}]\label{lem:(co)invariant}
The forgetful functor $U:\Stone_{\F_p}\to\Ring_{\F_p}$ admits right and left adjoint functors.
\end{lem}
\begin{proof}
By \cref{pro:p-Bool}, we identify $\Stone_{\F_p}$ with $\mathrm{Bool}_{\F_p}$. We explicitly obtain the desired right (resp. left) adjoint functor given by $\mathrm{ker}(\mathrm{Frob}-1)$ (resp. $\mathrm{coker}(\mathrm{Frob}-1)$).
\end{proof}

\begin{rmk}[{\cite[Proposition 1.6]{Gre24}}]
\cref{thm:Stone duality} gives the locally presentability of $\Stone_k$. So, even if we replace $\F_p$ by any ring $k$ whose idempotent elements are only $\{0,1\}$, the existence of a right adjoint functor for the forgetful functor $\Stone_k\to\Ring_k$ holds true.
\end{rmk}

\begin{dfn}[(co)invariant]\label{dfn:(co)invariant}
The right (resp. left) adjoint functor of the forgetful functor $\Stone_{\F_p}\to\Ring_{\F_p}$ is denoted by $(-)^{\mathrm{Frob}=1}$ (resp. $(-)_{\mathrm{Frob}=1}$), and called \emph{invariant} (resp. \emph{coinvariant}).
\exend
\end{dfn}

Finally, we relate $p$-Boolean algebras to perfect $\F_p$-algebras.

\begin{dfn}[{\cite[Definition 3.4.1]{Prism}}, (co)perfection]\label{dfn:(co)perfection}
The inclusion $\Perf_{\F_p}\subset\Ring_{\F_p}$ has a left (resp. right) adjoint functor $(-)_\text{perf}$ (resp. $(-)^\text{perf}$), which is called \emph{coperfection} (resp. \emph{perfection}):
\[
\xymatrix{
{\Perf_{\F_p}}\ar@{^{(}->}[rr]^-{\bot}_-{\bot}&&{\Ring_{\F_p}.}\ar@/^12pt/[ll]^-{\text{perfction}}\ar@/_12pt/[ll]_-{\text{coperfction}}
}
\]
The adjunction functors are given by the following sequential (co)limits:
\[
A_\text{perf}=\ilim(A\xrightarrow{\text{Frob}}A\xrightarrow{\text{Frob}}\cdots),\quad A^\text{perf}=\plim(\cdots\xrightarrow{\text{Frob}}A\xrightarrow{\text{Frob}}A).
\]
\exend
\end{dfn}

\begin{rmk}
Coperfection may be called \emph{direct limit perfection} (,for example, as in \cite{BS22}).
\end{rmk}

\begin{thm}[{\cite[Lemma 4.9]{Ant23}}]\label{thm:adjoint 1}
The following adjoint diagram exists.
\[
\xymatrix{
{\Stone_{\F_p}}\ar[rr]^-{\bot}_-{\bot}&&{\Perf_{\F_p}}\ar@/^12pt/[ll]^-{\text{invariant}}\ar@/_12pt/[ll]_-{\text{coinvariant}}\ar@{^{(}->}[rr]^-{\bot}_-{\bot}&&{\Ring_{\F_p}.}\ar@/^12pt/[ll]^-{\text{perfection}}\ar@/_12pt/[ll]_-{\text{coperfection}}
}
\]
\end{thm}
\begin{proof}
It follows from the fact that both Frobenius (co)invariant functors factor through $\Perf_{\F_p}$ since Stone $\F_p$-algebra is already perfect.
\end{proof}

\section{$\delta$-Rings}

In this section, we will review the notion of $\delta$-rings and define a new class "Stone $\delta$-rings".

\subsection{$\delta$-Rings}

We will examine the definition of $\delta$-rings and the rings of Witt vectors.

\begin{dfn}[{\cite[Definition 2.1 and Remark 2.2]{BS22}}, $\delta$-rings]\label{dfn:delta-ring}
A ring $A$ has a \emph{$\delta$-ring structure} on $A$ if $A$ admits a map $\delta:A\to A$ satisfying the following properties:
\begin{enumerate}
    \item We have $\delta(1)=0$.
    \item For any $x,y\in A$, the following holds:
    \[
    \delta(xy)=x^p\delta(y)+\delta(x)y^p+p\delta(x)\delta(y).
    \]
    \item For any $x,y\in A$, the following holds:
    \[
    \delta(x+y)=\delta(x)+\delta(y)+\dfrac{(x^p+y^p)-(x+y)^p}{p}.
    \]
\end{enumerate}
We often call $A$ by \emph{$\delta$-ring} if $A$ has a $\delta$-ring structure. Also, when $A$ has a $\delta$-ring structure, then the induced ring map $\varphi:A\to A$ is called a \emph{Frobenius lift}:
\[
\varphi(x):=x^p+p\delta(x).
\]
Furthermore, for any $\delta$-rings $A,B$, a ring map $f:A\to B$ is a \emph{$\delta$-ring map} if it satisfies $f\circ\delta_A=\delta_B\circ f$ (where $\delta_A$ (resp. $\delta_B$) denotes a $\delta$-ring structure on $A$ (resp. B)). We let $\Ring^\delta$ denote the category of $\delta$-rings with $\delta$-ring maps. 
\exend
\end{dfn}

\begin{pro}[{\cite[Remark 2.7]{BS22}}]
The category $\Ring^\delta$ has limits and colimits, and the forgetful functor $\Ring^\delta\to\Ring$ preserves limits and colimits. Moreover, the forgetful functor $\Ring^\delta\to\Ring$ admits right and left adjoint functors.
\end{pro}
\begin{proof}
In this article, we only prove the existence of adjoint functors. By the Freyd's adjoint functor theorem, it suffices to show that \cite[Exercise 2.5.12 and 2.5.14]{Prism}.

For the existence of the left adjoint functor, we prove \cite[Exercise 2.5.12]{Prism}. It suffices to show that for every ring $B$, every $\delta$-ring $A$ and every ring map $f:B\to A$, there exists a $\delta$-ring $A'$ such that $\#A'\leq\max\{\#B,\omega\}$ and $f$ factors through $A'$ in $\Ring^\delta$. The desired $\delta$-ring $A'$ is indeed the $\delta$-subring of $A$ generated by $f(B)$.

For the existence of the right adjoint functor, we prove \cite[Exercise 2.5.14]{Prism}. Similarly, it suffices to show that for every ring $B$, every $\delta$-ring $A$ and every ring map $f:A\to B$, there exists a $\delta$-ring $A'$ such that $\#A'\leq\max\{\#B,\omega\}$ and $f$ factors through $A'$ in $\Ring^\delta$. We define the ideal $I$ of $A$ as $\{x\in A\mid\delta^m(x)\in\mathrm{ker}(f)\ \text{for all}\ m\in\Z_{\geq0}\}$. Since we have $\delta(I)\subset I$, $A/I$ is a $\delta$-ring compatible with the $\delta$-structure of $A$ (by \cite[Lemma 2.9]{BS22}). We obtain a (set-theoretical) injection $A/I\hookrightarrow B\times B\times B\times B\times\cdots\ ;\ x\longmapsto(f(x),f(\delta(x)),\ldots)$, which implies that $A/I$ is the desired $\delta$-ring.
\end{proof}

\begin{dfn}[{\cite[Remark 2.7]{BS22}}, Witt vector functor and free $\delta$-rings]
The functor \emph{$W(-):\Ring\to\Ring^\delta$} denotes the right adjoint functor of the forgetful functor $\Ring^\delta\to\Ring$.
\exend
\end{dfn}

\begin{rmk}[{\cite[Remark 2.7]{BS22}} and {\cite[Remark 3.2.2]{Prism}}]
The right adjoint functor $W(-)$ of the forgetful functor coincides with the classical $p$-typical Witt vectors.
\exend
\end{rmk}

\subsection{Stone $\delta$-Rings}

We will discuss the relationship between $\delta$-rings and perfect $\F_p$-algebras.

\begin{dfn}[{\cite[Definition 2.26]{BS22}}, perfect $\delta$-rings]
A $\delta$-ring is \emph{perfect} if its Frobenius lift is an isomorphism.
\exend
\end{dfn}

\begin{rmk}\label{rmk:some analogue of char p}
Recall that there are the following properties for $\F_p$-algebras:
\begin{enumerate}
    \item An $\F_p$-algebra is reduced if and only if its Frobenius map is injective.
    \item An $\F_p$-algebra admits a surjection from a perfect $\F_p$-algebra if and only if its Frobenius map is surjective (i.e. it is \emph{semiperfect}).
\end{enumerate}
As the analogue, we have the following properties for $\delta$-rings:
\begin{enumerate}
    \item A $\delta$-ring is $p$-torsion-free if its Frobenius lift is injective (\cite[Lemma 2.28(1)]{BS22}).
    \item If an $\F_p$-algebra $A$ is reduced (; equivalently, $W(A)$ is reduced), then the Frobenius lift on $W(A)$ is injective (\cite[Remark 3.4.7]{Prism}).
    \item An $\F_p$-algebra $A$ is semiperfect if and only if the Frobenius lift on $W(A)$ is surjective (\cite[Remark 3.4.8]{Prism}).
\end{enumerate}
\exend
\end{rmk}

\begin{lem}[{\cite[Lemma 3.3.2]{Prism}}]\label{lem:strict p}
Let $A$ be a perfect $\F_p$-algebra. The ring $W(A)$ of Witt vectors is $p$-torsion-free and $p$-complete, and we have $W(A)/p\simeq A$.
\end{lem}

\cref{lem:strict p} and the deformation theory (the vanishing of the cotangent complex $L_{(-)/\F_p}$) gives an interesting relationship between the theory of characteristic $p$ and the theory of ($p$-complete) $\delta$-rings.

\begin{thm}[{\cite[Corollary 2.31]{BS22}}]\label{thm:Witt-mod p}
The following categories are equivalent:
\begin{enumerate}
    \item The category $\Perf^{\delta,\land}$ of $p$-complete perfect $\delta$-rings with ($p$-complete) $\delta$-ring maps.
    \item The category $\Perf_{\F_p}$ of perfect $\F_p$-algebras with $\F_p$-algebra maps.
\end{enumerate}
The functor from (1) to (2) is the modulo $p$ functor and the inverse functor is the Witt vector functor.
\end{thm}

Moreover, we can define (co)perfection for $\delta$-rings, which is an analogue of Frobenius-(co)perfection (\cref{dfn:(co)perfection}).

\begin{dfn}[(completed) $\delta$-(co)perfection]\label{dfn: delta-(co)perfection}
Let $A$ be a $p$-complete $\delta$-ring with a Frobenius lift $\varphi$. We define the following notions:
\begin{enumerate}
    \item The \emph{(completed) $\delta$-coperfection} $A_\text{perf}$ of $A$ is defined the $p$-completion of the sequential colimit $\ilim(A\xrightarrow{\varphi}A\xrightarrow{\varphi}\cdots)$.
    \item The \emph{(completed) $\delta$-perfection} $A^\text{perf}$ of $A$ is defined by the sequential limit $\plim(\cdots\xrightarrow{\varphi}A\xrightarrow{\varphi}A)$.
\end{enumerate}
\exend
\end{dfn}

\begin{pro}[(completed) $\delta$-(co)perfection adjunction]
The inclusion $\Perf^{\delta,\land}\subset\Ring^{\delta,\land}$ admits a left (resp. right) adjoint functor $(-)_\text{perf}$ (resp. $(-)^\text{perf}$) where $\Ring^{\delta,\land}$ denotes the category of $p$-complete $\delta$-rings with $p$-continuous $\delta$-ring maps.
\end{pro}
\begin{proof}
For any $p$-complete $\delta$-ring $A\in\Ring^{\delta,\land}$, we let $A'$ denote the sequential colimit $\ilim(A\xrightarrow{\varphi}A\xrightarrow{\varphi}\cdots)$. Since $A'$ is $p$-torsion-free and $\varphi_{A'}$ is an isomorphism, so it is perfect $\delta$-ring. The ring $A''$ denotes a $p$-completion of $A$. Note that $A''$ has a unique $\delta$-ring structure compatible with the $\delta$-ring structure of $A$ by \cite[Lemma 2.17]{BS22}, and we have $A''\simeq W(A''/p)$.

Note that $p$-completion is a left adjoint functor of the inclusion $\Ring\subset\Ring^\land$, we obtain the desired functor. Also, the claim for the right adjoint follows from the similar argument. (cf. \cite[Remark 2.27 and Lemma 3.9]{BS22})
\end{proof}

\begin{dfn}[Stone $\delta$-rings]\label{dfn:Stone delta-rings}
We let $\Stone^{\delta,\land}$ denote the essential image of $W:\mathrm{Bool}_{\F_p}\to\Perf^{\delta,\land}$. An object of $\Stone^{\delta,\land}$ is called a \emph{Stone $\delta$-ring}.
\exend
\end{dfn}

We will examine the new class $\Stone^{\delta,\land}$ of $\delta$-rings as the analogue of the relationship between $\Stone_{\F_p}$ and $\Ring_{\F_p}$.

\begin{pro}
A $p$-complete $\delta$-ring is Stone if and only if its Frobenius lift is the identity map.
\end{pro}
\begin{proof}
If a $p$-complete $\delta$-ring $A$ is Stone, there exists a Stone $\F_p$-algebra $R$ such that $A$ is isomorphic to $W(R)$. So, we can describe the action of the Frobenius lift by using Teichm\"{u}ller expansions:
\[
\varphi\left(\sum_{n=0}^\infty [a_n]p^n\right)=\sum_{n=0}^\infty[a_n^p]p^n=\sum_{n=0}^\infty[a_n]p^n\quad\text{where each $a_n\in R$}.
\]
Consequentially, we obtain the desired property.

Conversely, for a given $p$-complete $\delta$-ring $A$ with the trivial Frobenius lift, the modulo $p$ reduction $A/p$ is a Stone $\F_p$-algebra. Note that $A$ has no $p$-torsion element by \cref{rmk:some analogue of char p}, the ring $W(A/p)\simeq A$ of Witt vectors is Stone.
\end{proof}

Next, we can define the following functors $\Ring^{\delta,\land}\to\Stone^{\delta,\land}$ as an analogue of the Frobenius-(co)invariant functor (\cref{dfn:(co)invariant}).

\begin{dfn}[(completed) $\delta$-(co)invariant]\label{dfn:delta-(co)invariant}
The right (resp. left) adjoint functor of the inclusion $\Stone^{\delta,\land}\subset\Ring^{\delta,\land}$ is denoted by $(-)^{\varphi=1}$ (resp. $(-)_{\varphi=1}^\land$), and called \emph{(completed) $\delta$-invariant} (resp. \emph{(completed) $\delta$-coinvariant}).
\exend
\end{dfn}

\begin{rmk}
Similarly as in \cref{lem:(co)invariant}, the (completed) $\delta$-(co)invariant functors are described explicitly:
\[
A^{\varphi=1}=\mathrm{ker}(\varphi-1:A\to A),\quad A_{\varphi=1}^\land=\mathrm{coker}(\varphi-1:A\to A)^\land,
\]
where $\varphi:A\to A$ denotes a Frobenius lift on $A$. (Note that only $\delta$-coinvariant functor is needed the $p$-completion.)
\exend
\end{rmk}

\begin{thm}[Stone duality 2]\label{thm:Stone duality 2}
There exists an equivalence of categories:
\[
(\Stone^{\delta,\land})^\text{op}\simeq\mathrm{ProFin}.
\]
\end{thm}
\begin{proof}
By \cref{thm:Stone duality}, it suffices to show $\Stone_{\F_p}\simeq\Stone^{\delta,\land}$. Indeed, the desired equivalence can be realized by the functors:
\[
\Phi:\Stone^{\delta,\land}\hookrightarrow\Perf^{\delta,\land}\xrightarrow{\mod p}\Perf_{\F_p},\quad\Psi:\Stone_{\F_p}\xrightarrow{\text{forgetful}}\Perf_{\F_p}\xrightarrow{W}\Perf^{\delta,\land}.
\]
Note that $\Phi$ and $\Psi$ are well-defined by \cref{dfn:Stone delta-rings}, and $\mod p$ and $W$ give an equivalence at the level of $p$-complete perfect $\delta$-rings by \cref{thm:Witt-mod p}, so the functors above gives the desired equivalence.
\end{proof}

\begin{lem}\label{lem:Stone duality}
Let $S$ be a profinite set. There exists an isomorphism:
\[
W(\mathrm{Cont}(S,\F_p))\simeq\mathrm{LocConst}(S,\Z_p)^\land\simeq\mathrm{Cont}(S,\Z_p).
\]
\end{lem}
\begin{proof}
We have the following isomorphism:
\[
W(\mathrm{Cont}(S,\F_p))/p\simeq\mathrm{Cont}(S,\F_p)\simeq\mathrm{Cont}(S,\Z_p)/p\simeq\mathrm{LocConst}(S,\Z_p)/p.
\]
(The second isomorphism can be checked pointwise.)

Note that both sides are $p$-torsion-free, so we obtain the desired isomorphism from \cref{thm:Witt-mod p}.
\end{proof}

\begin{rmk}\label{rmk:Stone duality}
By comparing isomorphisms of the "original" Stone duality (\cref{thm:Stone duality}), we can also realize the following explicit equivalence:
\[
(\Stone^{\delta,\land})^\text{op}\ni A\longmapsto\Hom_\delta^\land(A,\Z_p)\in\mathrm{ProFin},\quad\mathrm{ProFin}\ni S\longmapsto\mathrm{Cont}(S,\Z_p)\in (\Stone^{\delta,\land})^\text{op},
\]
where $\Hom^{\delta,\land}(A,\Z_p)$ denotes the (profinite) set of $p$-completely $\delta$-maps from $A$ to $\Z_p$.
\exend
\end{rmk}

\begin{rmk}
The following diagram represents the relationship discussed above:
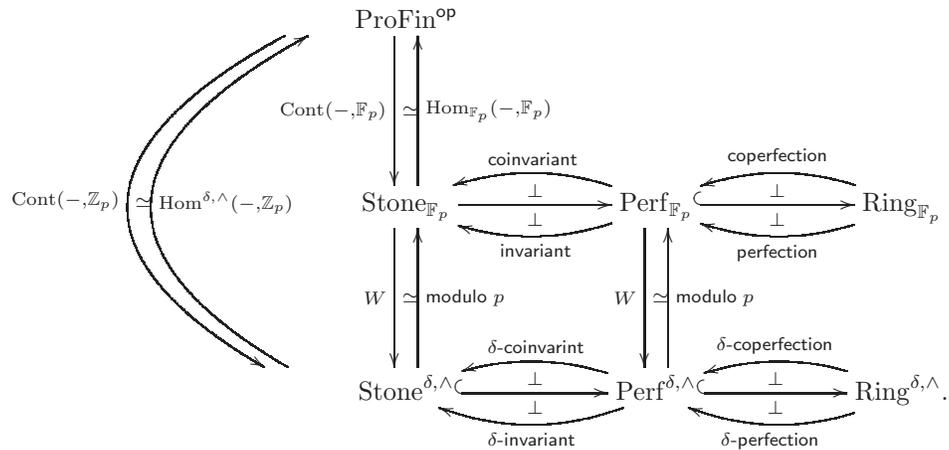
\[
\xymatrix{
{\mathrm{ProFin}^\text{op}}\ar@<-1ex>[dd]_-{\mathrm{Cont}(-,\F_p)}^-{\simeq}\ar@<-8ex>@/_70pt/[dddd]_-{\mathrm{Cont}(-,\Z_p)}^-{\simeq}&&&& \\
&&&& \\
{\Stone_{\F_p}}\ar[rr]^-{\bot}_-{\bot}\ar@<-1ex>[dd]_-{W}^-{\simeq}\ar@<-1ex>[uu]_-{\Hom_{\F_p}(-,\F_p)}&&{\Perf_{\F_p}}\ar@<-1ex>[dd]_-{W}^-{\simeq}\ar@/^12pt/[ll]^-{\text{invariant}}\ar@/_12pt/[ll]_-{\text{coinvariant}}\ar@{^{(}->}[rr]^-{\bot}_-{\bot}&&{\Ring_{\F_p}}\ar@/^12pt/[ll]^-{\text{perfection}}\ar@/_12pt/[ll]_-{\text{coperfection}} \\
&&&&\\
{\Stone^{\delta,\land}}\ar@<6ex>@/^70pt/[uuuu]_-{\Hom^{\delta,\land}(-,\Z_p)}\ar@{^{(}->}[rr]^-{\bot}_-{\bot}\ar@<-1ex>[uu]_-{\text{modulo $p$}}&&{\Perf^{\delta,\land}}\ar@<-1ex>[uu]_-{\text{modulo $p$}}\ar@/^12pt/[ll]^-{\text{$\delta$-invariant}}\ar@/_12pt/[ll]_-{\text{$\delta$-coinvarint}}\ar@{^{(}->}[rr]^-{\bot}_-{\bot}&&{\Ring^{\delta,\land}.}\ar@/^12pt/[ll]^-{\text{$\delta$-perfection}}\ar@/_12pt/[ll]_-{\text{$\delta$-coperfection}}
}
\]
\exend
\end{rmk}

\section{Light Condensed Mathematics}

In this section, we will review the notions of light profinite sets and discuss the definition and properties of light condensed sets. 

\subsection{Light Profinite Sets}

We collect the well-known results for topological spaces to define the notion of "light" profinite sets and compact Hausdorff spaces.

\begin{thm}[Urysohn's metrization theorem]\label{thm:Urysohn}
Let $K$ be a compact Hausdorff space. The following are equivalent.
\begin{enumerate}
    \item The space $K$ is metrizable.
    \item The space $K$ is second countable.
    \item There exists an embedding $K\hookrightarrow[0,1]^\mathbb{N}$ into the Hilbert cube.
\end{enumerate}
\end{thm}

In terms of profinite sets, we have more characterization of matrizableness. 

\begin{lem}[{\cite[Proposition 2.1.2]{Solid}}]\label{lem:light}
Let $S$ be a profinite set. The following are equivalent.
\begin{enumerate}
    \item The space $S$ is metrizable.
    \item The space $S$ is second countable.
    \item The space $S$ can be written as a sequential limit of finite sets.
    \item The set $\mathrm{Cont}(S,\Z)$ is countable.
    \item The set $\mathrm{Cont}(S,\F_2)$ is countable.
\end{enumerate}
\end{lem}

\begin{dfn}[{\cite[Definition 2.1.3]{Solid}}, light profinite set]
A profinite set is \emph{light} if it satisfies the equivalent conditions of \cref{lem:light}. We let $\mathrm{ProFin}^\text{light}$ denote the category of light profinite sets.

Similarly, a compact Hausdorff space is \emph{light} if it satisfies the equivalent conditions of \cref{thm:Urysohn}. We let $\mathrm{CHaus}^\text{light}$ denote the category of light compact Hausdorff spaces.
\exend
\end{dfn}

\begin{exa}[{\cite[Example 2.1.8]{Solid}}]\label{exa:sequence and Cantor set}
Here are important examples of light profinite sets.
\begin{enumerate}
    \item $\tilde{\mathbb{N}}:=\mathbb{N}\cup\{\infty\}=\plim[n]\{1,2,\ldots,n,\infty\}$ is light. For $n\leq m$, the map $\{1,2,\ldots,m,\infty\}\to\{1,2,\ldots,n,\infty\}$ in the system is given by sending $i$ to $i$ for $i\leq n$, and to $\infty$ otherwise.
    
    For a first-countable topological space $X$ and $x\in X$, a continuous map $\tilde{\mathbb{N}}\to X$ can be regarded as a sequence $x_1,x_2,\ldots$ that “converges” to  $x\in X$ by the first countability of $X$.
    \item The Cantor set $\{0,1\}^\mathbb{N}=\plim[S\subset\mathbb{N}]\{0,1\}^S$ is light. Note that the Cantor set is not countable.
    
    For a profinite set, it is light if and only if there is a surjection from the Cantor set (\cite[Proposition 2.19]{Solid}).
\end{enumerate}
\exend
\end{exa}

There is also a connection with Stone $\F_p$-algebras.

\begin{rmk}[{\cite[Proposition 4.2]{Gre24}}, Stone duality in light setting]\label{rmk:light Stone duality}

The functor as in \cref{thm:Stone duality} “preserves lightness”. That is, there exists the following equivalence of categories:
\[
\Stone_{\F_p,\aleph_1}\simeq(\mathrm{ProFin}^\text{light})^\text{op}.
\]
Here, $\aleph_1$ is the successor of the countable cardinal $\aleph_0$, and $\Stone_{\F_p,\aleph_1}$ denotes a full subcategory of $\aleph_1$-small Stone $\F_p$-algebras; equivalently, the category of countable Stone $\F_p$-algebras.
\end{rmk}

\begin{rmk}[light profinite covers]\label{rmk:light profinite cover}
For a light profnite set $S\in\mathrm{ProFin}^\text{light}$, the Stone-\v{C}ech compactification $\beta(S)$ is not necessarily light. 

Similarly, a light compact Housdorff space $K$ admits a "profinite cover" by taking the Stone-\v{C}ech compactification $\beta(K)$. However, we cannot take the cover as a "light profinite cover".

Indeed, we can take the following another cover as a "light profinite cover" (see also the proof of \cite[Proposition 2.2.8]{Solid}):

For a countable basis $\mathcal{B}$ of $K$ (we fix it), the index set $I$ of finite covers of $K$ in $\mathcal{B}$ is a countable cofiltered set. For any $i\in I$, if we let $S_i$ denote $\{U_{i,j_1},\ldots,U_{i,j_i}\}$, then we obtain the following surjection:
\[
f:S(=\plim[i\in I]S_i)\twoheadrightarrow K\ ;\ u=\{U_{i,k_i}\}_{i\in I}\longmapsto\bigcap_{i\in I}U_{i,k_i}.
\]
In fact, the set $f(u)$ is necessarily a point. If we pick two different points $x,y\in \bigcap_{i\in I}U_{i,k_i}$, then we take two closed neighborhoods $Z$ of $x$ and $Z'$ of $y$ with no intersection by Hausdorffness, which give a cover $\{U:=K\setminus Z, U:=K\setminus Z'\}$ of $K$. By definition, there exists a cover $S_{i_0}$ such that it is a refinement of $\{U,U'\}$. Since we have $x\in\bigcap_{i\in I}U_{i,k_i}$, $U_{i_0,k_{i_0}}\subset U'$, thus $y$ does not lie in $U_{i_0,j_{i_0}}$, which gives contradiction. Also, we can show non-emptiness of $\bigcap_{i\in I}U_{i,k_i}$ by the finite intersection property.
\exend
\end{rmk}

\begin{pro}[{\cite[Proposition 2.1.4]{Solid}}]\label{pro:replete}
The category $\mathrm{ProFin}^\text{light}$ admits countable limits. Moreover, a sequential limit of surjections is a surjection. (i.e. for any light profinite set $S=\plim[n]S_n$ with surjective transition maps, each projection $S\to S_n$ is a surjection.)
\end{pro}

\begin{rmk}[{\cite{Gle58}}, extremally disconnected]
In particular, a finite product of light profinite sets is light again. However, a finite product of extremally disconnected spaces (i.e. projective objects in $\mathrm{CHaus}$) is not necessarily extremally disconnected. 
\exend
\end{rmk}

\subsection{Light Condensed Sets}

Next, we will define light condensed mathematics, in particular, light condensed sets.

\begin{dfn}[{\cite[Definition 2.2.1]{Solid}}, light condensed mathematics]
Let $C$ be a category with filtered colimits. The \emph{category of light condensed objects} of $C$ is the following category (topos):
\[
\Cond(C)^\text{light}:=\mathrm{Shv}_\text{eff}(\mathrm{ProFin}^\text{light},C).
\]
Here, $\mathrm{ProFin}^\text{light}$ is endowed with the finitary Grothendieck topology by finite jointly surjective covers (,noting that the notion of \emph{eff}ective epimorphisms is identified with the notion of surjetions.)
\exend
\end{dfn}

\begin{nota}
We abbreviate $\Cond(\mathrm{Set})^\text{light}$ as $\Cond$.
\end{nota}

\begin{rmk}
In light setting, we do not have to worry about set-theorical issue for $\mathrm{ProFin}$, however the light setting contains some practical setting. So, "it makes our discussion light".
\exend
\end{rmk}

\begin{exa}[{\cite[Example 2.2.3]{Solid}}, condensification]\label{exa:condensification}
Let $T$ be a topological space. The functor
\[
\underline{T}:\mathrm{ProFin}^\text{light}\ni S\longmapsto\mathrm{Cont}(S,T)\in\mathrm{Set}
\]
is a light condensed set since a surjection from a compact Hausdorff space is a quotient map.

This construction defines a functor $\mathrm{Top}\to\Cond$, which is called \emph{condensification}. From now, we will abbreviate an image $\underline{T}$ of condensification as $T$ for any topological space if there is no confusion.
\exend
\end{exa}

\begin{pro}[{\cite[Proposition 2.2.4]{Solid}}]\label{pro:underlying topological space}
The condensification functor $\underline{(-)}:\mathrm{Top}\to\Cond$ admits a left adjoint functor $(-)(\ast)_\text{top}$, which can be given by $X(\ast)_\text{top}=\colim[\underline{S}\to X]S\quad\text{in}\ \mathrm{Top}$ ; equivalently, the following topology:
\begin{enumerate}
    \item The underlying set is $X(\ast)_\text{top}=X(\ast)$.
    \item The topology is that for any surjection $\pi:\underline{S}\twoheadrightarrow X$  from $\mathrm{ProFin}^\text{light}$, $\pi(\ast)^{-1}(U)\subset S$ is an open subset in $S$ if $U\subset X(\ast)$.
\end{enumerate}
\end{pro}

\begin{rmk}[replete topoi]\label{rmk:replete}
For any family $\{X_n\}_{n\in\mathbb{N}}$ of light condensed sets with surjective transition maps and compatible surjections $X_n\twoheadrightarrow Y$ in $\Cond$, the map $X:=\plim[n]X_n\to Y$ is a surjection.

In fact, it suffices to show that for any light profinite set $S$ and any element $y\in Y(S)$, there exists a surjection $T\twoheadrightarrow S$ from a light profinite set such that $y$ can lift to an element $x\in X(T)$. For the surjection $X_0\twoheadrightarrow Y$, there exists a surjection $T_0\twoheadrightarrow S$ from a light profinite set such that $y$ can lift to an element $x_0\in X_0(T_0)$. Inductively, we obtain the sequence $\{x_n\}_{n\in\mathbb{N}}$ and $\{T_n\}_{n\in\mathbb{N}}$, which gives the desired lift $x\in X(\plim[n]T_n)$ of $y$. (Note that $\plim[n]T_n\to S$ is a surjection by \cref{pro:replete})

The observation makes $\Cond$ replete in the sense of \cite[Definition 3.1.1]{BS15}. For example, the notion of replete gives the exactness of countable products in $\Cond$ by \cite[Proposition 3.19]{BS15}.
\exend
\end{rmk}

Finally, we examine some relationship between topological spaces and light condensed sets. This discussion is followed by \cite[Section 3.1]{Cas24} which discuss general topoi but not light condensed sets. 

\begin{dfn}[{\cite[Definition 3.1.15]{Cas24}} and {\cite[Definition 2.2.6 and Remark 2.2.7]{Solid}}, qc, qs and qcqs in (general) topoi]
We define the following notions:
\begin{enumerate}
    \item A light condensed set $X$ is \emph{quasicompact (qc)} if for any family $\{X_i\}_{i\in I}$ with a surjection $\bigsqcup_{i\in I}X_i\twoheadrightarrow X$, there exists a finite subset $J$ of $I$ such that $\bigsqcup_{j\in J}X_j\to X$  is a surjection.     
    \item A light condensed set $X$ is \emph{quasiseparated (qs)} if for any diagram $U\to X\leftarrow V$ from quasicompact objects, the fiber product $U\times_XV$ is quasicompact.
    \item A light condensed set $X$ is \emph{qcqs} if $X$ is quasicompact and quasiseparated.
\end{enumerate}
\exend
\end{dfn}

\begin{pro}[characterization of qc and qs in $\Cond$]\label{pro:qc and qs}
Let $X$ be a light condensed set. The following properties are satisfied:
\begin{enumerate}
    \item The object $X$ is qc if and only if there exists a surjection $S\twoheadrightarrow X$ from (Yoneda's image of) a light profinite set $S$.
    \item The object $X$ is qs if and only if for any diagram $S_1\to X\leftarrow S_2$ from light profinite sets, $S_1\times_XS_2\subset S_1\times S_2$ is a light profinite set.
\end{enumerate}
\end{pro}
\begin{proof}
(1) The object $X$ can be written as a colimit $\colim[i\in I]S_i$ of light profinite sets. We obtain a surjection $\bigsqcup_{i\in I}S_i\twoheadrightarrow\colim[i\in I]S_i\simeq X$. If $X$ is qc, there exists a finite subset $J\subset I$ such that $\bigsqcup_{j\in J}S_j\to X$ is a cover of $X$ again. This gives the desired surjection from a light profinite set.

Conversely, we assume that $X$ admits a surjection $S\twoheadrightarrow X$ from a light profinite set. For any cover $\bigsqcup_{i\in I}U_i\twoheadrightarrow X$ of $X$, we obtain a new cover $S\times\bigsqcup_{i\in I}U_i\twoheadrightarrow S$ of $S$. Since $S$ is qc, there exists a finite subset $J\subset I$ such that $S\times\bigsqcup_{j\in J}U_j\twoheadrightarrow S$ is a cover of $S$ again. Therefore, we obtain the desired finite cover $\bigsqcup_{j\in J}U_j\twoheadrightarrow X$.

(2) For any diagram $S_1\to X\leftarrow S_2$ in $\Cond$ from any light profinite sets, we have $S_1\times_XS_2$ is a light condensed subset of $S_1\times S_2$. Therefore, the desired characterization of qs follows from \cref{rmk:qc} below.
\end{proof}

\begin{rmk}\label{rmk:qc}
Let $S$ be a light profinite set, and let $X$ be a light condensed subset of $S$. The following are equivalent:
\begin{enumerate}
    \item The object $X$ is qc.
    \item The object $X$ is represented by a closed subset of $S$.
    \item The object $X$ is a light profinite set.
\end{enumerate}
In fact, noting that $(2)\Rightarrow(3)$ and $(3)\Rightarrow(1)$ follows clearly, it suffices to show the implication $(1)\Rightarrow(2)$. It 
essentially follows from \cite[Lemma 2.7]{CC}.
\exend
\end{rmk}

\begin{thm}[{\cite[Proposition 2.2.8]{Solid}}]\label{thm:qcqs}
The condensification functor induces an equivalence $\mathrm{CHaus}^\text{light}\xrightarrow{\sim}\Cond_\text{qcqs}$. Moreover, there exists an equivalence $\mathrm{Ind}_{\text{inj}}(\mathrm{CHaus}^\text{light})\simeq\Cond_\text{qs}$ (by \cite[Proposition 3.1.24]{Cas24}).
\end{thm}

\section{Comparisons of Sites}

In this section, we will compare Grothendieck topologies in $\mathrm{ProFin}^\text{light}$, $\Stone_{\F_p}$ and $\Stone^{\delta,\land}$.

\subsection{Faithfully Flat Covers}

we will review the comparison between the fpqc topology on $\Ring_{\F_p}$ and the topology on $\mathrm{ProFin}$ in \cite{Gre24}.

\begin{lem}[{\cite[Proposition 1.4 and Lemma 1.6]{Gre24}} and {\cite[\href{https://stacks.math.columbia.edu/tag/092N}{[092N]}]{Sta}}]\label{lem:weakly etale}
Let $k$ be a ring. A Stone $k$-algebra is a weakly \'{e}tale $k$-algebra. Moreover, a map $A\to B$ of Stone $k$-algebras is faithfully flat if and only if the induced map $\mathrm{Spec}(B)\to\mathrm{Spec}(A)$ of spectra is a surjection.
\end{lem}

So, we obtain the isomorphism of sites between profinite sets and Stone $k$-algebras.

\begin{thm}[{\cite[Proposition 1.20]{Gre24}}]\label{thm:comparison of sites}
Let $k$ be a field. There is an isomorphism of sites:
\[
(\mathrm{ProFin},\tau_\text{eff})\simeq((\Stone_k)^\text{op},\tau_\text{fpqc}),
\]
where $\tau_\text{eff}$ (resp. $\tau_\text{fpqc}$) denotes the Grothendieck topology defined by finite jointly surjective covers (resp. faithfully flat covers).
\end{thm}

\begin{pro}[{\cite[Proposition 2.16]{Gre24}}]
Let $k$ be a field. A profinite set $S$ corresponds to $\mathrm{Spec}(\mathrm{Cont}(S,k))$ as the fpqc sheaf.
\end{pro}

\subsection{$p$-Completely Faithfully Flat Covers}

We will introduce the notion of completely flatness, and compare the some topologies.

\begin{dfn}[{\cite[Section 1.2]{BS22}}, completely (faithfully) flat maps]
Let $A$ be a ring and $I$ be a finitely generated ideal of $A$. A ring map $A\to B$ is \emph{$I$-completely (faithfully) flat} if $B\otimes_A^LA/I$ is concentrated in degree $0$, and the induced map $A/I\to B/I$ is (faithfully) flat. If $I$ is generated by a single element $d$, we abbreviate $(d)$-completely flatness as $d$-completely flatness.
\exend
\end{dfn}

\begin{lem}\label{lem:cover comparison}
The $p$-completely faithful flatness in $\Perf^{\delta,\land}$ coincides with the faithful flatness in $\Perf_{\F_p}$ via the equivalence in \cref{thm:Witt-mod p}.
\end{lem}
\begin{proof}
By definition, a $p$-completely faithfully flat map $A\to B$ in $\Perf^{\delta,\land}$ induces a faithfully flat map $A/p\to B/p$ in $\Perf_{\F_p}$.

For a faithfully flat map $A\to B$ in $\Perf_{\F_p}$, the induced map $W(A)/p\simeq A\to W(B)/p\simeq B$ is clearly faithfully flat, and we have 
\[
W(B)\otimes_{W(A)}^L W(A)/p\simeq W(B)/p,
\]
noting that $W(A)$ and $W(B)$ are $p$-torsion-free. So, the canonical map $W(A)\to W(B)$ is $p$-completely faithfully flat.
\end{proof}

\begin{thm}[the comparison of sites]\label{thm:comparison of sites 2}
There is an isomorphism of sites as follows:
\[
(\mathrm{ProFin},\tau_\text{eff})\simeq((\Stone^{\delta,\land})^\text{op},\tau_{p\text{-fpqc}}),
\]
where the Grothendieck topology $\tau_{p\text{-fpqc}}$ means the $p$-completely faithfully flat topology.
\end{thm}
\begin{proof}
By \cref{thm:comparison of sites}, it suffices to show the following isomorphism of sites:
\[
((\Stone_k)^\text{op},\tau_\text{fpqc})\simeq((\Stone^{\delta,\land})^\text{op},\tau_{p\text{-fpqc}}).
\]
Note that $\Stone^{\delta,\land}$ is isomorphic to $\Stone_{\F_p}$ by \cref{dfn:Stone delta-rings}, \cref{lem:cover comparison} gives the desired isomorphism.
\end{proof}

Also, we can obtain a characterization of Stone $\delta$-rings.

\begin{pro}
A $p$-complete $\delta$-ring $A$ is Stone if and only if the canonical map $A\to A_{\varphi=1}^\land$ is $p$-completely faithfully flat.
\end{pro}
\begin{proof}
Obviously, the canonical map $ A\to A_{\varphi=1}^\land$ for any Stone $\delta$-ring $A$ is the identity map, in particular, a $p$-completely faithfully flat map.

Conversely, noting that $A_{\varphi=1}^\land$ is $p$-torsion-free, $A$ has no $p$-torsion element if the canonical map $A\to A_{\varphi=1}^\land$ is $p$-completely faithfully flat. Modulo $p$ reduction $A/p\to (A/p)_{\varphi/p=1}$ is faithfully flat, in particular, injective and surjective (i.e. bijective). So, $A/p$ is a Stone $\F_p$-algebra, which gives desired property for $A\simeq W(A/p)$.
\end{proof}

\section{Connection between $\delta$-Rings and Light Condensed Sets}

In this section, we will summarize the above discussion, and conclude the connection between $\delta$-rings and light condensed sets.

\subsection{Totally Disconnected Formal $\delta$-Stacks}

We examine the relation to light condensed sets according to \cite[Subsection V.3.1]{Sch24}. (We only focus on light setting, but some results hold for general setting.)

\begin{rmk}\label{rmk:light 1}
For any countably presented Stone $\delta$-ring $A$ (i.e. Stone $\delta$-ring which is a countably presented $\Z_p$-algebra), the corresponding profinite set $S$ is light. In fact, noting that $S$ is given by $\mathrm{Spec}(A/p)$ and $A/p$ is countably presented over $\F_p$, the lightness of $S$ follows from \cref{rmk:light Stone duality}.
\exend
\end{rmk}

\begin{dfn}[totally disconnected formal $\delta$-stacks]
We let $\mathrm{TD}^{\delta,\land}$ denote the opposite category of the category of countably presented  Stone $\delta$-rings, and $\mathrm{TD}\mathrm{Stk}^{\delta,\land}$ denote the category of sheaves on $\mathrm{TD}^{\delta,\land}$ with $p$-completely faithfully flat topology:
\[
\mathrm{TD}\mathrm{Stk}^{\delta,\land}:=\mathrm{Shv}_{p\text{-fpqc}}(\mathrm{TD}^{\delta,\land},\mathrm{Set}).
\]
We call an object in $\mathrm{TD}\mathrm{Stk}^{\delta,\land}$ by a \emph{totally disconnected formal $\delta$-stack}.
\exend
\end{dfn}

This object defined above is the analogue of fpqc stacks discussed in \cite{Gre24}.

\begin{rmk}\label{rmk:the terminology}
We work in the set-valued sheaves because the author would like to clarify the topological aspects of $\delta$-rings. Also, as in \cite[Remark 3.9]{Gre24}, there would be no problem with this terminology "stack".  
\exend
\end{rmk}

\subsection{Relation to Light Condensed Sets}

We give examples of $p$-complete $\delta$-rings from light condensed sets as an analogue of \cref{lem:Stone duality}. 

\begin{pro}\label{pro:Gelfand duality}
Let $K$ be a light compact Hausdorff space. The ring $\mathrm{Cont}(K,\Z_p)$ is a $p$-complete perfect $\delta$-ring.
\end{pro}
\begin{proof}
It suffices to show that $\mathrm{Cont}(K,\Z_p)$ is $p$-torsion-free and $p$-complete, and $\mathrm{Cont}(K,\F_p)$ is a perfect $\F_p$-algebra. For the first part, the $p$-torsion-freeness follows from the fact that $\Z_p$ has no $p$-torsion, and the $p$-completeness follows from $\mathrm{Cont}(K,\Z_p)\simeq(\mathrm{LocConst}(K,\Z_p))^\land$. For the second part, we obviously obtain the $p$-th root of $f(x)$ for any function $f:K\to\F_p$ and $x\in K$, which gives the map $K\ni x\longmapsto f(x)^{1/p}\in\F_p$. It suffices to show that the map is continuous. If we take a surjection $S\twoheadrightarrow K$ from a light profinite set (for example, \cref{rmk:light profinite cover}), the composition $S\to\F_p$ is continuous by \cref{lem:Stone duality}. Note that the map $S\twoheadrightarrow K$ is a quotient map, we obtain the desired continuity.
\end{proof}

Finally, we discuss a relationship between totally disconnected formal $\delta$-stacks and light condensed set.

\begin{dfn}[{\cite[\href{https://stacks.math.columbia.edu/tag/00WV}{[00WV]}]{Sta}}, continuous functors]
Let $C,D$ be sites, and let $u:C\to D$ be a functor. The functor $u$ is \emph{continuous} if for any cover $\{U_i\to U\}_{i\in I}$ in $C$, the family $\{u(U_i)\to u(U)\}_{i\in I}$ is a cover in $D$, and for any map $V\to U$ in $C$, each map $u(V)\times_{u(U)}u(U_i)\to u(V\times_UU_i)$ is an isomorphism. 
\exend
\end{dfn}

\begin{pro}\label{lem:continuous}
The functor $u:\mathrm{TD}^{\delta,\land}\ni A\longmapsto\Hom^{\delta,\land}(A,\Z_p)=\mathrm{Spec}(A/p)\in\mathrm{ProFin}^\text{light}$ preserves finite limits and covers. In particular, $u$ is continuous.

Moreover, $u$ sends a final object in $\mathrm{TD}^{\delta,\land}$ to a final object $\ast$ in $\mathrm{ProFin}^\text{light}$.
\end{pro}
\begin{proof}
The claim for finite limits follows clearly, noting that $\mathrm{TD}$ is defined as the opposite category. The claim for covers follows from the previous claims (\cref{rmk:Stone duality} and \cref{rmk:light 1}). The final claim for final objects is obtained from $u(\Z_p)=\mathrm{Spec}(\F_p)=\ast$.
\end{proof}

\begin{cor}\label{cor:morphism of topoi}
There exists a morphism $\psi:\Cond\to\mathrm{TD}\mathrm{Stk}^{\delta,\land}$ of topoi such that it is induced by $u:\mathrm{TD}^{\delta,\land}\to\mathrm{ProFin}^\text{light}$. Explicitly, $\psi$ is given by the (pair of the) following functors of topoi:
\begin{enumerate}
    \item For any light condensed set $X\in\Cond$, we have $\psi_\ast X(A)\simeq X(u(A))$.
    \item The functor $\psi^\ast$ is the sheafification of the functor $u_p$ (defined in \cite[\href{https://stacks.math.columbia.edu/tag/00VC}{[00VC]}]{Sta}).
\end{enumerate}
\end{cor}
\begin{proof}
\cref{lem:continuous} and \cite[\href{https://stacks.math.columbia.edu/tag/00WX}{[00WX]} and \href{https://stacks.math.columbia.edu/tag/00X6}{[00X6]}]{Sta} give the desired properties.
\end{proof}

\begin{exa}\label{exa:Betti delta-stacks}
We examine the totally disconnected formal $\delta$-stack $\psi_\ast K\in\mathrm{TD}\mathrm{Stk}^{\delta,\land}$ where $K$ is a light compact Hausdorff space (, noting that $K$ can be considered as a qcqs light condensed set). This example is inspired by \cite[Example 11.12]{Sch22}.

First, when we assume that $K$ is light profinite, we see that $\psi_\ast K$ is isomorphic to $\mathrm{Cont}(K,\Z_p)$. In fact, by \cref{cor:morphism of topoi}, 
for any countably presented Stone $\delta$-ring $A\in\mathrm{TD}^{\delta,\land}$, there is the following isomorphism:
\[
(\psi_\ast K)(A)\simeq\underline{K}(\pi(A))=\mathrm{Cont}(K,\Hom^{\delta,\land}(A,\Z_p))=\Hom_\Cond(K,\Hom^{\delta,\land}(A,\Z_p))\simeq\Hom^{\delta,\land}(\mathrm{Cont}(K,\Z_p),A).
\]
We let $K_\text{Betti}^\delta$ denote the totally disconnected formal $\delta$-stack $\psi_\ast K$.

Next, assuming that $K$ is light compact Hausdorff, we can take a light profinite cover $S\twoheadrightarrow K$ and a (effective) equivalence relation $R=S\times_K S\subset S\times S$ such that $K$ is homeomorphic to $S/R$ (by \cref{rmk:light profinite cover} and \cref{thm:qcqs}). We claim that $\psi_\ast K\simeq S_\text{Betti}^\delta/R_\text{Betti}^\delta$. Note that, by \cref{cor:morphism of topoi}, the functor $\psi_\ast$ admits a right adjoint functor, we have $S_\text{Betti}^\delta/R_\text{Betti}^\delta\simeq\psi_\ast(\underline{S}/\underline{R})$. Also, by the proof of \cref{thm:qcqs}, $\underline{K}$ is isomorphic to $\underline{S}/\underline{R}$. Therefore, we obtain the desired isomorphism:
\[
\psi_\ast K\simeq\psi_\ast(\underline{S}/\underline{R})\simeq S_\text{Betti}^\delta/R_\text{Betti}^\delta.
\]
\exend
\end{exa}


\begin{thebibliography}{1}
\bibitem[Ant23]{Ant23} Benjamin Antieau, {\em Spherical Witt vectors and integral models for spaces}, \url{https://arxiv.org/pdf/2308.07288} (2023).
\bibitem[BS15]{BS15} Bhargav Bhatt and Peter Scholze, {\em The pro-\'{e}tale topology for schemes}, Ast\'{e}risque \textbf{369} (2015), 99-201.
\bibitem[BS22]{BS22} Bhargav Bhatt and Peter Scholze, {\em Prisms and Prismatic Cohomology}, Annals of Mathematics, \textbf{196}(3) (2022), 1135–1275.
\bibitem[Cas24]{Cas24} Attilio Castano, {\em A Berkovich Approach to Perfectoid Spaces}, \url{https://arxiv.org/pdf/2304.09266} (2024).
\bibitem[CS19]{LCM} Dustin Clausen and Peter Scholze, {\em Lectures on Condensed Mathematics}, available at the website (\url{https://people.mpim-bonn.mpg.de/scholze/Condensed.pdf}), (2019).
\bibitem[CS22]{CC} Dustin Clausen and Peter Scholze, {\em Condensed Mathematics and Complex Geometry}, available at the website (\url{https://people.mpim-bonn.mpg.de/scholze/Complex.pdf}), (2022).
\bibitem[CS24]{AnStk} Dustin Clausen and Peter Scholze, {\em Analytic Stacks}, available at YouTube (\url{https://youtube.com/playlist?list=PLx5f8IelFRgGmu6gmL-Kf_Rl_6Mm7juZO&si=hEDdnB9eKfvka-pn}), (2024).
\bibitem[Solid]{Solid} Juan Esteban Rodr\'{i}guez Camargo, {\em Notes on solid geometry}, available at the website (\url{https://blogs.cuit.columbia.edu/jr4460/seminar-on-solid-geometry/}), (2024).
\bibitem[Gle58]{Gle58} Andrew Mattei Gleason, {\em Projective Topological Spaces}, Illinois J.Math. \textbf{2} (1958), 482-489.
\bibitem[Gre24]{Gre24} Rok Gregoric, {\em Stone duality between condensed mathematics and algebraic geometry}, \url{https://arxiv.org/pdf/2401.02568}  (2024).
\bibitem[Prism]{Prism} Kiran Sridhara Kedlaya, {\em Notes on prismatic cohomology}, available at the website (\url{https://kskedlaya.org/prismatic/prismatic.html}), (2024).
\bibitem[Stacks]{Sta} The Stacks Project Authors, {\em Stacks Project}, \url{https://stacks.math.columbia.edu}.
\bibitem[Sch22]{Sch22} Peter Scholze, {\em \'{E}tale cohomology of diamonds}, \url{https://arxiv.org/pdf/1709.07343} (2022).
\bibitem[Sch24]{Sch24} Peter Scholze, {\em Geometrization of the real local Langlands correspondence}, available at the website (\url{https://people.mpim-bonn.mpg.de/scholze/RealLocalLanglands.pdf}), (2024).
\bibitem[Sto34]{Sto34} Marshall Harvey Stone, {\em Boolean Algebras and Their Application to Topology}, Proc Natl Acad Sci U S A. \textbf{20} (3) (1934), 197–202.
\bibitem[Sto36]{Sto36} Marshall Harvey Stone, {\em The theory of representations for Boolean algebras}, Trans. Amer. Math. Soc, \textbf{40}
(1936), 37–111.
\end{thebibliography}
\end{document}